\documentclass[12pt, reqno]{amsart}
\usepackage{fullpage,url,amssymb}
\usepackage[backref]{hyperref}
\usepackage[alphabetic,backrefs,lite]{amsrefs}

\usepackage{color}
\newtheorem{lemma}{Lemma}[section]

\newtheorem{claim*}{Claim}

\newtheorem{thm}[lemma]{Theorem}

\newcommand{\Aff}{{\mathbb A}}

\newcommand{\PP}{{\mathbb P}}

\newcommand{\F}{{\mathbb F}}
\newcommand{\Q}{{\mathbb Q}}

\newcommand{\Z}{{\mathbb Z}}

\newcommand{\calE}{{\mathcal E}}

\newcommand{\calL}{{\mathcal L}}

\newcommand{\calO}{{\mathcal O}}

\DeclareMathOperator{\Sym}{Sym}

\DeclareMathOperator{\Jac}{Jac}

\numberwithin{equation}{section}
\numberwithin{table}{section}

\newcommand{\defi}[1]{\textsf{#1}} % for defined terms

\title{A family of varieties with exactly one pointless rational fiber}

\author[B. Viray]{Bianca Viray}
\address{Department of Mathematics, University of California, Berkeley, CA 94720-3840, USA}
\email{bviray@math.berkeley.edu}
\urladdr{http://math.berkeley.edu/~bviray}
\date{}
\thanks{This research was supported by the Mentored Research Award from UC Berkeley}

\begin{document}
	
	%%%%%%%%%%%%%%%%%%%%%%%%%%%%%%%%%%%%%%%%%%%%%%%%%%%%%%%%%%%%%%%%%%%%%%%%%
	\begin{abstract}
		We construct a concrete example of a $1$-parameter family of smooth 
		projective geometrically integral varieties over an open subscheme of 
		$\PP^1_{\Q}$ such that there is exactly one rational fiber with no 
		rational points.   This makes explicit a construction of Poonen.
	\end{abstract}
	%%%%%%%%%%%%%%%%%%%%%%%%%%%%%%%%%%%%%%%%%%%%%%%%%%%%%%%%%%%%%%%%%%%%%%%%%

	\maketitle

	%%%%%%%%%%%%%%%%%%%%%%%%%%%%%%%%%%%%%%%%%%%%%%%%%%%%%%%%%%%%%%%%%%%%%%%%%
	\section{Introduction}%%%%%%%%%%%%%%%%%%%%%%%%%%%%%%%%%%%%%%%%%%%%%%%%%%%
	%%%%%%%%%%%%%%%%%%%%%%%%%%%%%%%%%%%%%%%%%%%%%%%%%%%%%%%%%%%%%%%%%%%%%%%%%

		%We present an explicit algebraic family which provides ``an extreme example of geometry \emph{not} controlling arithmetic''~\cite[p.2]{poonen-chatelet}.
		%where the geometry strikingly fails to control the arithmetic.  
		We construct a family of 
		smooth projective geometrically integral surfaces over an open 
		subscheme of $\PP^1_{\Q}$ with the following curious arithmetic 
		property: there is exactly one $\Q$-fiber with no rational points.  Our 
		proof makes explicit a non-effective construction of 
		Poonen~\cite[Prop. 7.2]{poonen-chatelet}, thus giving ``an extreme 
		example of geometry \emph{not} controlling 
		arithmetic''~\cite[p.2]{poonen-chatelet}.  
		We believe that this is the first example of its kind.
	
		\begin{thm}\label{thm:main}
			Define $P_0(x):=(x^2-2)(3-x^2)$ and $P_\infty(x) := 2x^4 + 3x^2 - 
			1$.  Let $\pi\colon X\to \PP^1_{\Q}$ be the Ch\^atelet surface 
			bundle over $\PP^1_{\Q}$ given by
			\[
				y^2 + z^2 = \left(6u^2 - v^2\right)^2P_0(x) + 
					\left(12v^2\right)^2P_{\infty}(x),
			\]
			where $\pi$ is projection onto $(u:v)$.  Then $\pi(X(\Q)) = 
			\mathbb{A}^1_{\Q}(\Q)$.
		\end{thm}	

		Note that the degenerate fibers of $\pi$ do not lie over $\PP^1(\Q)$ so 
		the family of smooth projective geometrically integral surfaces 
		mentioned above contains all $\Q$-fibers.
		
		The non-effectivity in~\cite[Prop. 7.2]{poonen-chatelet} stems from the use of higher genus curves and Faltings' theorem.  (This is described in more detail in~\cite[\S9]{poonen-chatelet}).  We circumvent the use of higher genus curves by an appropriate choice of $P_{\infty}(x)$.

		% We construct a family of 
		% smooth projective geometrically integral surfaces over an open 
		% subscheme of $\PP^1_{\Q}$ with the following curious arithmetic 
		% property: there is exactly one $\Q$-fiber with no rational points.  This is ``an extreme example of geometry \emph{not} controlling arithmetic''~\cite[p.2]{poonen-chatelet}.  
		% We believe that this is the first example of its kind.
		% 	
		% \begin{thm}\label{thm:main}
		% 	Define $P_0(x):=(x^2-2)(3-x^2)$ and $P_\infty(x) := 2x^4 + 3x^2 - 
		% 	1$.  Let $\pi\colon X\to \PP^1_{\Q}$ be the Ch\^atelet surface 
		% 	bundle over $\PP^1_{\Q}$ given by
		% 	\[
		% 		y^2 + z^2 = \left(6u^2 - v^2\right)^2P_0(x) + 
		% 			\left(12v^2\right)^2P_{\infty}(x),
		% 	\]
		% 	where $\pi$ is projection onto $(u:v)$.  Then $\pi(X(\Q)) = 
		% 	\mathbb{A}^1_{\Q}(\Q)$.
		% \end{thm}	
		% 
		% Our proof makes explicit a non-effective construction of 
		% Poonen~\cite[Prop. 7.2]{poonen-chatelet}.  The non-effectivity 
		% in~\cite[Prop. 7.2]{poonen-chatelet} stems from the use of higher 
		% genus curves and Faltings' theorem.  We circumvent the use of higher 
		% genus curves by an appropriate choice of $P_{\infty}(x)$.

	%%%%%%%%%%%%%%%%%%%%%%%%%%%%%%%%%%%%%%%%%%%%%%%%%%%%%%%%%%%%%%%%%%%%%%%%%
	\section{Background}%%%%%%%%%%%%%%%%%%%%%%%%%%%%%%%%%%%%%%%%%%%%%%%%%%%%%
	%%%%%%%%%%%%%%%%%%%%%%%%%%%%%%%%%%%%%%%%%%%%%%%%%%%%%%%%%%%%%%%%%%%%%%%%%

		This information can be found in~\cite[\S3,5, and 6]{poonen-chatelet}.  
		We review it here for the reader's convenience.

		Let $\mathcal E$ be a rank $3$ vector sheaf on a $k$-variety $B$.  A 
		\defi{conic bundle} $C$ over $B$ is the zero locus in $\PP\calE$ of a 
		nowhere vanishing zero section $s \in \Gamma(\PP\mathcal E, \Sym^2(\mathcal E))$.  A \defi{diagonal conic bundle} is a conic bundle where 
		$\mathcal E = \mathcal L_1 \oplus\mathcal L_2 \oplus \mathcal L_3$ and 
		$s = s_1 + s_2 + s_3,  s_i \in \Gamma(\PP\mathcal E, \calL_i^{\otimes2})$.
	
		Now let $\alpha \in k^\times,$ and let $P(x) \in k[x]$ be a separable 
		polynomial of degree $3$ or $4$.  Consider the diagonal conic bundle 
		$X$ given by $B = \PP^1, \mathcal L_1 = \mathcal O, \mathcal L_2 = 
		\mathcal O, \mathcal L_3 = \mathcal O(2), s_1 = 1, s_2 = -\alpha, s_3 
		= -w^4P(x/w)$.  This smooth conic bundle contains the affine 
		hypersurface $y^2 - \alpha z^2 = P(x) \subset \Aff^3$ as an open 
		subscheme.  We say that $X$ is the \defi{Ch\^atelet surface} given by
		\[
			y^2 - \alpha z^2 = P(x).
		\]
		Note that since $P(x)$ is not identically zero, $X$ is an integral 
		surface.
	
		A \defi{Ch\^atelet surface bundle over} $\PP^1$ is a flat proper 
		morphism $V \to \PP^1$ such that the generic fiber is a Ch\^atelet 
		surface.  We can construct them in the following way.  Let $P, Q \in 
		k[x,w]$ be linearly independent homogeneous polynomials of degree $4$ 
		and let $\alpha \in k^\times$.  Let $V$ be the diagonal conic bundle 
		over $\PP^1_{(a:b)} \times \PP^1_{(w:x)}$ given by $\mathcal L_1 = 
		\calO, \calL_2 = \calO, \calL_3 = \calO(1,2), s_1 = 1, 
		s_2 = -\alpha, s_3 = -(a^2P + b^2Q)$.  By composing $V \to 
		\PP^1\times\PP^1$ with the projection onto the first factor, we 
		realize $V$ as a Ch\^atelet surface bundle.  We say that $V$ is the 
		Ch\^atelet surface bundle given by 
		\[
			y^2 - \alpha z^2 = a^2P(x) + b^2Q(x),
		\]
		where $P(x) = P(x,1)$ and $Q(x) = Q(x,1)$.  We can also view 
		$a,b$ as relatively prime, homogeneous, degree $d$ polynomials in $u, 
		v$ by pulling back by a suitable degree $d$ map $\phi\colon 
		\PP^1_{(u:v)} \to\PP^1_{(a:b)}$.

	%%%%%%%%%%%%%%%%%%%%%%%%%%%%%%%%%%%%%%%%%%%%%%%%%%%%%%%%%%%%%%%%%%%%%%%%%
	\section{Proof of Theorem~\ref{thm:main}}%%%%%%%%%%%%%%%%%%%%%%%%%%%%%%%%
	%%%%%%%%%%%%%%%%%%%%%%%%%%%%%%%%%%%%%%%%%%%%%%%%%%%%%%%%%%%%%%%%%%%%%%%%%
		By~\cite{iskovskikh-chatelet}, we know that the Ch\^atelet surface
		\[
			y^2+z^2=(x^2-2)(3-x^2)
		\]
		violates the Hasse principle, i.e. it has $\Q_v$-rational points for all 
		completions $v$, but no $\Q$-rational points.  Thus, $\pi(X(\Q)) 
		\subseteq \mathbb{A}^1_{\Q}(\Q)$.  
		Therefore, it remains to show that $X_{(u:1)}$, the Ch\^atelet surface 
		defined by
		\[
			y^2+z^2 = (6u^2-1)^2P_0(x) + 12^2P_\infty(x),
		\]
		has a rational point for all $u\in \Q$.  
	
		If $P_{(u:1)}:=(6u^2-1)^2P_0(x) + 12^2P_\infty(x)$ is irreducible, 
		then by~\cite{ctssd-chatelet},~\cite{ctssd-chatelet2} we know that 
		$X_{(u:1)}$ satisfies the Hasse principle.  Thus it suffices to show 
		that $P_{(u:1)}$ is irreducible and $X_{(u:1)}(\Q_v)\neq \emptyset$ 
		for all $u\in\mathbb Q$ and all places $v$ of $\Q$.
	
		%%%%%%%%%%%%%%%%%%%%%%%%%%%%%%%%%%%%%%%%%%%%%%%%%%%%%%%%%%%%%%%%%%%%%%
		\subsection{Irreducibility}%%%%%%%%%%%%%%%%%%%%%%%%%%%%%%%%%%%%%%%%%%%
		%%%%%%%%%%%%%%%%%%%%%%%%%%%%%%%%%%%%%%%%%%%%%%%%%%%%%%%%%%%%%%%%%%%%%%
		
			We prove that for any $u\in \Q$, the polynomial 
			$P_{(u:1)}\left(x\right)$ is 
			irreducible in $\Q[x]$ by proving the slightly more general 
			statement, that for all $t\in\Q$
			\[
				P_t(x) := (2x^4 + 3x^2 - 1) + t^2(x^2 - 2)(3 - x^2) 
					= x^4(2 - t^2) + x^2(3 + 5t^2) + (-6t^2 - 1)
			\]
			is irreducible in $\Q[x]$.	We will use the fact that if $a,b,c \in 
			\Q$ are such that $b^2-4ac$ and $ac$ are not squares in $\Q$ then 
			$p(x):=ax^4+bx^2+c$ is irreducible in $\Q[x]$.

			Let us first check that for all $t \in \Q$, 
			$\left(3 + 5t^2\right)^2 - 4\left(2 - t^2\right)
			\left(-6t^2 - 1\right)$ is not a square in $\Q$.  
			This is equivalent to proving that the affine curve $C\colon 
			w^2 = t^4 + 74t^2 + 17$ has no rational points.  The smooth 
			projective model, $\overline{C} : w^2 = t^4 + 74t^2s^2 + 17s^4$ in 
			weighted projective space $\PP(1,1,2)$, has $2$ rational points at 
			infinity.  Therefore $\overline{C}$ is isomorphic to its Jacobian.
			A computation in \texttt{Magma} shows that 
			$\Jac(C)(\Q) \cong \Z/2\Z$~\cite{magma}.  Therefore, the points at 
			infinity are the only $2$ rational points of $\overline{C}$ and thus 
			$C$ has no rational points.

			Now we will show that $\left(-6t^2 - 1\right)\left(2 - t^2\right)$ 
			is not a square in $\Q$ for any $t\in \Q$.  As above, this is 
			equivalent to determining whether $C'\colon w^2=(-6t^2-1)(2-t^2)$ 
			has a rational point.  Since $6$ is not a square in $\Q$, this is 
			equivalent to determining whether the smooth projective model, 
			$\overline{C'}$, has a rational point.  The curve $\overline{C'}$ is 
			a genus $1$ curve so it is either isomorphic to its Jacobian or has 
			no rational points.  A computation in \texttt{Magma} shows that 	
			$\Jac\left(C'\right)\left(\Q\right)\cong \Z/2\Z$~\cite{magma}.  Thus 
			$\#C'\left(\mathbb{Q}\right)=0$ or $2$.  If $\left(t,w\right)$ is a 
			rational point of $C'$, then $\left(\pm t,\pm w\right)$ is also a 
			rational point.	Therefore, 	$\#C\left(\mathbb{Q}\right)=2$ if and 
			only if there is a point with $t = 0$ or $w = 0$ and one can easily 
			check that this is not the case.

		%%%%%%%%%%%%%%%%%%%%%%%%%%%%%%%%%%%%%%%%%%%%%%%%%%%%%%%%%%%%%%%%%%%%%%
		\subsection{Local Solvability}%%%%%%%%%%%%%%%%%%%%%%%%%%%%%%%%%%%%%%%%
		%%%%%%%%%%%%%%%%%%%%%%%%%%%%%%%%%%%%%%%%%%%%%%%%%%%%%%%%%%%%%%%%%%%%%%

			\begin{lemma}\label{lem:some-local-pts}
				For any point $(u:v)\in \mathbb{P}^1_{\mathbb{Q}}$, the 
				Ch\^atelet surface $X_{(u:v)}$ has $\mathbb{R}$-points and 
				$\mathbb{Q}_p$-points for every prime $p$.
			\end{lemma}

			\begin{proof}
				Let $a = 6u^2-v^2$ and let $b=12v^2$.  We will refer to 
				$a^2 P_0(x) + b^2 P_{\infty}(x)$ both as $P_{(a:b)}$ and 
				$P_{(u:v)}$.
		
				$\mathbb{R}$-points: It suffices to show that given $(u: v)$ 
				there exists an $x$ such that 
				\[
					P_{(a: b)}=x^4(2b^2-a^2)+x^2(3b^2+5a^2)+(-6a^2-b^2)
				\] 
				is positive.  If $2b^2 - a^2$ is positive, then any $x$ 
				sufficiently large will work.  So assume $2b^2 - a^2$ is 
				negative.  Then $\alpha=\frac{-(3b^2 + 5a^2)}{2(2b^2 - a^2)}$ is 
				positive.  We claim $P_{(a: b)}(\sqrt{\alpha})$ is positive.
				\begin{eqnarray*}
					P_{(a: b)}(\sqrt{\alpha}) & = & \alpha^2(2b^2 - a^2) 
						+ \alpha(3b^2 + 5a^2) + (-6a^2 - b^2)\\
					& = & \frac{(3b^2 + 5a^2)^2}{4(2b^2 - a^2)} 
						+ \frac{-(3b^2 + 5a^2)^2}{2(2b^2 - a^2)}
						+ (-6a^2 - b^2)\\
					& = & \frac{1}{4(2b^2 - a^2)}
						\left( 4(2b^2 - a^2)(-6a^2 - b^2)
						- (3b^2 + 5a^2)^2\right)\\
					& = & \frac{1}{4(2b^2 - a^2)}
					\left(-17b^4 - 74a^2b^2 - a^4\right)
				\end{eqnarray*} 
				Since $2b^2-a^2$ is negative by assumption and $-17b^4 - 
				74a^2b^2 - a^4$ is always negative, we have our result.
		
				$\mathbb{Q}_p$-points:  
				\begin{description}
					
					\item[$p\geq 5$] 
						Without loss of generality, let $a$ and 
						$b$ be relatively prime integers.  Let 
						$\overline{X}_{(a:b)}$ denote the reduction of $X_{(a 
						: b)}$ modulo $p$.  We claim that there exists a smooth 
						$\F_p$-point of $\overline{X}_{(a:b)}$ that, by Hensel's 
						lemma, we can lift to a $\Q_p$-point of $X_{(a:b)}$.  
						
						Since $P_{(a:b)}$ has degree at 
						most $4$ and is not identically zero modulo $p$, there 
						is some $x\in \F_p$ such that $P_{(a:b)}\left(x\right)$ 
						is nonzero.  Now let $y, z$ run over all values in 
						$\F_p$.  Then the polynomials $y^2, 
						P_{(a:b)}\left(x\right) - z^2$ each take $(p + 1)/2$ 
						distinct values.  By the pigeonhole principle, $y^2$ and  
						$P_{(a:b)}\left(x\right) - z^2$ must agree for at least 
						one pair $(y,z) \in \F_p^2$ and one can check that this 
						pair is not $(0,0)$.  Thus, this tuple $(x,y,z)$ gives a 
						smooth $\F_p$-point of $\overline{X}_{(a:b)}$.  (The 
						proof above that the quadratic form $y^2 + z^2$ 
						represents any element in $F_p$ is not new. For example, 
						it can be found in~\cite[Prop
						 5.2.1]{Cohen-numbertheoryI}.)

						% We will show that $y^2 = P_{(a:b)}\left(x\right) - z^2$ has a solution in $F_p$ which we can lift to a solution over $\Q_p$ using Hensel's lemma.  
						% 
						% 
						% For this value of $x$, the equation 
						% $y^2 + z^2 = P_{(a:b)}\left(x\right)w^2$ 
						% 
						% defines a 
						% smooth genus $0$ projective curve over $\F_p$.  By the 
						% Weil conjectures, this curve has exactly $p + 1$ points 
						% over $\F_p$, so there is at least one $F_p$-point with 
						% $w = 1$.  Using Hensel's lemma, we can lift this point 
						% to a $\mathbb{Q}_p$-point.  

					\item[$p=3$] 
						From the equations for $a$ and $b$, one can check that 
						for any $(u: v) \in \mathbb{P}^1_{\mathbb{Q}}$, 
						$v_3(b/a)$ is positive.  Since 
						$\mathbb{Q}_3(\sqrt{-1})/\mathbb{Q}_3$ is an 
						unramified extension, it suffices to show that given 
						$a,b$ as above, there exists an $x$ such that $P_{(a: 
						b)}(x)$ has even valuation.  Since $v_3(b/a)$ is 
						positive, $v_3(2b^2 - a^2)=2v_3(a)$.  Therefore, if 
						$x=3^{-n}$, for $n$ sufficiently large, the valuation 
						of $P_{(a: b)}(x)$ is $-4n+2v_3(a)$ which is even. 
		
					\item[$p=2$] 
						From the equations for $a$ and $b$, one can check that 
						for any $(u: v) \in \mathbb{P}^1_{\mathbb{Q}}$, 
						$v_2(b/a)$ is at least $2$.  Let $x=0$ and $y=a$.  
						Then we need to find a solution to $z^2 = 
						a^2(-7+(b/a)^2)$.  Since $v_2(b/a)>1$, 
						$-7+(b/a)^2\equiv1^2\mod 8$.  By Hensel's lemma, we 
						can lift this to a solution in $\Q_2$.
				\end{description}
			\end{proof}
	
	%%%%%%%%%%%%%%%%%%%%%%%%%%%%%%%%%%%%%%%%%%%%%%%%%%%%%%%%%%%%%%%%%%%%%%%%%
	\section*{Acknowledgements}
		I thank my advisor, Bjorn Poonen, for suggesting the problem and many 
		helpful conversations.  I thank the referee for suggesting the
		pigeonhole principle argument in Lemma~\ref{lem:some-local-pts}.  I 
		also thank Daniel Erman and the referee for comments improving the 
		exposition.
	%%%%%%%%%%%%%%%%%%%%%%%%%%%%%%%%%%%%%%%%%%%%%%%%%%%%%%%%%%%%%%%%%%%%%%%%%

	%%%%%%%%%%%%%%%%%%%%%%%%%%%%%%%%%%%%%%%%%%%%%%%%%%%%%%%%%%%%%%%%%%%%%%%%
	%%%%%%%%%%%%%%%   		Bibliography				%%%%%%%%%%%%%%%%%%%%
	%%%%%%%%%%%%%%%%%%%%%%%%%%%%%%%%%%%%%%%%%%%%%%%%%%%%%%%%%%%%%%%%%%%%%%%%
	

\begin{bibdiv}
		\begin{biblist}

			\bib{magma}{article}{
			   author={Bosma, Wieb},
			   author={Cannon, John},
			   author={Playoust, Catherine},
			   title={The Magma algebra system. I. The user language},
			   note={Computational algebra and number theory (London, 1993)},
			   journal={J. Symbolic Comput.},
			   volume={24},
			   date={1997},
			   number={3-4},
			   pages={235--265},
			   issn={0747-7171},
			   review={\MR{1484478}},
			}
			
			\bib{Cohen-numbertheoryI}{book}{
			   author={Cohen, Henri},
			   title={Number theory. Vol. I. Tools and Diophantine equations},
			   series={Graduate Texts in Mathematics},
			   volume={239},
			   publisher={Springer},
			   place={New York},
			   date={2007},
			   pages={xxiv+650},
			   isbn={978-0-387-49922-2},
			   review={\MR{2312337 (2008e:11001)}},
			}
			

			\bib{ctssd-chatelet}{article}{
			   author={Colliot-Th{\'e}l{\`e}ne, Jean-Louis},
			   author={Sansuc, Jean-Jacques},
			   author={Swinnerton-Dyer, Peter},
			   title={Intersections of two quadrics and Ch\^atelet surfaces.
			 		I},
			   journal={J. Reine Angew. Math.},
			   volume={373},
			   date={1987},
			   pages={37--107},
			   issn={0075-4102},
			   review={\MR{870307 (88m:11045a)}},
			}

			\bib{ctssd-chatelet2}{article}{
			   author={Colliot-Th{\'e}l{\`e}ne, Jean-Louis},
			   author={Sansuc, Jean-Jacques},
			   author={Swinnerton-Dyer, Peter},
			   title={Intersections of two quadrics and Ch\^atelet surfaces. 
					II},
			   journal={J. Reine Angew. Math.},
			   volume={374},
			   date={1987},
			   pages={72--168},
			   issn={0075-4102},
			   review={\MR{876222 (88m:11045b)}},
			}

			\bib{iskovskikh-chatelet}{article}{
			   author={Iskovskih, V. A.},
			   title={A counterexample to the Hasse principle for systems of 
					two quadratic forms in five variables},
			   language={Russian},
			   journal={Mat. Zametki},
			   volume={10},
			   date={1971},
			   pages={253--257},
			   issn={0025-567X},
			   review={\MR{0286743 (44 \#3952)}},
			}

			\bib{poonen-chatelet}{article}{
			   author={Poonen, Bjorn},
			   title={Existence of rational points on smooth projective 
					varieties},
			   journal={J. Eur. Math. Soc. (JEMS)},
			   volume={11},
			   date={2009},
			   number={3},
			   pages={529--543},
			   issn={1435-9855},
			   review={\MR{2505440}},
			}

		\end{biblist}
	\end{bibdiv}
\end{document}